\documentclass[letterpaper,fleqn]{article}

\usepackage[margin=1.50in]{geometry}
\usepackage{marginnote}

\usepackage{mathptmx}
\usepackage[T1]{fontenc}
\usepackage{euscript}

\usepackage{latexsym,epsfig,verbatim}
\usepackage{amsmath,amsthm,amssymb}
\usepackage{colonequals}
\usepackage{units}
\usepackage{leftidx}

\usepackage{url}
\usepackage{color}
\usepackage[colorlinks,citecolor=blue,linkcolor=red!80!black]{hyperref}
\usepackage[nameinlink]{cleveref}
\usepackage{comment}

\usepackage{accents}

\usepackage{tikz}
\usetikzlibrary{matrix,arrows,decorations.pathmorphing}
\usepackage{pinlabel}

\theoremstyle{plain}
\newtheorem{theorem}{Theorem}[section]
\newtheorem{lemma}[theorem]{Lemma}
\newtheorem{corollary}[theorem]{Corollary}

\crefname{fact}{Fact}{Facts}

\crefname{claim}{Claim}{Claims}
\newtheorem*{claim*}{Claim}
\theoremstyle{definition}
\newtheorem{definition}[theorem]{Definition}

\crefname{question}{Question}{Questions}

\newtheorem{counterexample}[theorem]{Counter Example}

\newcommand{\define}[1]{\textbf{#1}}

\newcommand{\R}{\mathbb{R}}
\newcommand{\Z}{\mathbb{Z}}

\newcommand{\N}{\mathbb{N}}

\newcommand{\norm}[1]{\left\Vert#1\right\Vert}
\newcommand{\abs}[1]{\left\vert#1\right\vert}
\newcommand{\inv}{^{-1}}

\DeclareMathOperator{\Out}{Out}
\DeclareMathOperator{\Aut}{Aut}

\DeclareMathOperator{\rank}{rank}
\DeclareMathOperator{\conv}{Conv}
\newcommand{\hull}[1]{{\EuScript H}\!(#1)}
\newcommand{\free}{{F}} % the base free group (of rank n)
\renewcommand{\int}{\mathcal{I}} % the intersection graph

\newcommand{\nbhd}[2]{{\EuScript N}_{#1}(#2)} % basically the same thing as a ball, but #2 need not be a point
\newcommand{\dhaus}{d_{\mathrm{Haus}}} % notation for the Hausdorff distance
\newcommand{\len}[1]{\ell(#1)} % length (i.e., cardinality) of a tuple
\newcommand{\magnitude}[1]{\norm{#1}} % magnitude of a tuple
\newcommand{\conjmag}[1]{{\EuScript C}(#1)} % conjugacy magnitute of a tuple
\newcommand{\underlie}[1]{{\EuScript U}\!(#1)} % underlying tuple of a partitioned tuple

\newcommand{\J}{\mathbf{J}}

\newcommand{\cay}[2]{\mathrm{Cay}({#2}, {#1})} % The #1--Cayley graph of a f.g group #2=<#1>

\begin{document}

\renewcommand{\thefootnote}{\fnsymbol{footnote}} 
\renewcommand{\thefootnote}{\arabic{footnote}} 

\title{Rank and Nielsen equivalence in hyperbolic extensions}

\author{Spencer Dowdall and Samuel J. Taylor
\thanks{
The first named author was supported by NSF grant DMS-1711089; the second named author was supported by NSF grants DMS-1400498 and DMS-1744551.}}
\date{\today}

\maketitle

\begin{abstract}
In this note, we generalize a theorem of Juan Souto on rank and Nielsen equivalence in the fundamental group of a hyperbolic fibered $3$--manifold to a large class of hyperbolic group extensions. This includes all hyperbolic extensions of surfaces groups as well as hyperbolic extensions of free groups by convex cocompact subgroups of $\Out(F_n)$. 
\end{abstract}

\section{Introduction}
Perhaps the most basic invariant of a finitely generated group is its \define{rank}, that is, the minimal cardinality of a generating set. Despite its simple definition, rank is notoriously difficult to calculate even for well-behaved groups. For example, work of Baumslag, Miller, and Short \cite{BMS} shows that the rank problem is unsolvable for hyperbolic groups. In this note we calculate the rank for a large class of hyperbolic group extensions and furthermore show that, up to Nielsen equivalence, all minimal (i.e., minimal-cardinality) generating sets are of a standard form. 

Let $1 \to H \to G \to \Gamma \to 1$ be an exact sequence of infinite hyperbolic groups. We say that the extension has 
the \define{Scott--Swarup property} if each finitely generated, infinite index subgroup of $H$ is quasiconvex as a subgroup of $G$. Every subgroup $\Delta \le \Gamma$ induces a new short exact sequence $1\to H\to G_\Delta\to \Delta\to 1$, where $G_\Delta$ is the full preimage of $\Delta$ under the surjection $G\to \Gamma$. 
Our main theorem is the following; for the statement $\ell_\Gamma(\cdot)$ denotes conjugacy length with respect to any finite generating set for $\Gamma$.

\begin{theorem}\label{th:intro1}
Let $1 \to H \to G \to \Gamma \to 1$ be an exact sequence of infinite hyperbolic groups that has 
the Scott--Swarup property and torsion-free kernel $H$. For every $r\ge0$ there is an $N \ge 0$ such that 
if $\Delta \le \Gamma$ is a finitely generated subgroup with $\rank(\Delta) \le r$
and $\ell_\Gamma(\delta) \ge N$ for each $\delta \in \Delta \setminus \{1\}$, then 
\[
\rank(G_\Delta) = \rank(H) + \rank(\Delta).
\]
Moreover, every minimal generating set for $G_\Delta$ is Nielsen equivalent to a generating set which contains a minimal generating set for $H$ and projects to a minimal generating set for $\Delta$.
\end{theorem}

Examples of subgroups $\Delta\le \Gamma$ satisfying these conditions can easily be constructed. Indeed, for any set $\delta_1,\dotsc,\delta_r$ of pairwise independent infinite order elements of $\Gamma$, \Cref{th:intro1} applies to $\Delta = \langle \delta_1^m,\dotsc, \delta_r^m\rangle$ for all sufficiently large $m$. Alternately, one can build finite-index subgroups $K\le \Gamma$ such that \Cref{th:intro1} applies to every rank $r$ subgroup of $K$. 

\Cref{th:intro1} generalizes a theorem of Juan Souto \cite{souto2008rank}, who established this result when $\Gamma\cong \Z$ and $H$ is the fundamental group of a closed orientable surface $S_g$ of genus $g\ge 2$. 
Here the extension is induced by a hyperbolic $S_g$--bundle over $S^1$ with pseudo-Anosov monodromy $f\colon S_g\to S_g$, so that  $G$ is the fundamental group of the mapping torus $M_f$ of $f$.
In this language, Souto proves that the rank of $\pi_1(M_{f^N})\cong G_{\langle f^N \rangle}$ is equal to $2g+1$ for $N$ sufficiently large. Moreover, any two minimal generating sets in this situation are Nielsen equivalent. See also the work of Biringer--Souto \cite{BS_rank} for more on this special case.
In this paper, we use techniques previously established by Kapovich and Weidmann \cite{KapovichWeidmann-kleinian,KapovichWeidman--indecomp} to generalize Souto's result to \Cref{th:intro1}.

\Cref{th:intro1} applies to all hyperbolic extensions of surface groups \cite{FarbMosher, H, KentLein} as well as all hyperbolic extensions of free groups by convex cocompact subgroups of $\Out(F_n)$ \cite{DT1, HaHe, DT2}. We thus obtain the following corollary:

\begin{corollary}
\label{cor:applications}
The conclusions of \Cref{th:intro1} hold for all extensions of the following forms:
\begin{enumerate}\renewcommand{\theenumi}{\roman{enumi}}
\item\label{surface} Extensions $1\to \pi_1(S_g)\to G\to \Gamma\to 1$ with $G$ and $\Gamma$ both infinite and hyperbolic.
\item\label{free} Extensions $1\to \free_g\to G\to \Gamma\to 1$  such that $G$ is hyperbolic and the induced outer action $\Gamma\to \Out(\free_g)$ has convex cocompact image.
\end{enumerate}
\end{corollary}
\begin{proof}
Since the kernels of the above extensions are torsion-free, it suffices to verify the Scott--Swarup property. For the surface group extensions in (\ref{surface}), this was established by Scott and Swarup in the case that $\Gamma\cong \Z$ \cite{scott1990geometric} and by Dowdall--Kent--Leininger in the general case \cite{dowdall2014pseudo} (see also \cite{MjRafi}). For the free group extensions in (\ref{free}), Mitra \cite{mitra1999theorem} established the Scott--Swarup property when $\Gamma\cong \Z$ and the general case was proven by the authors in \cite{DT2} and by Mj--Rafi in \cite{MjRafi}.
\end{proof}

We note that Souto's theorem is exactly case (\ref{surface}) above with $\Gamma$ a cyclic group; the other cases of \Cref{cor:applications} are all new. In particular, the result is new even for free-by-cyclic groups $G = \free\rtimes_\phi \Z$ with fully irreducible and atoroidal monodromy $\phi\in\Out(\free_g)$, where the conclusion is that $\free_g\rtimes_{\phi^N}\Z$ has rank $g+1$ for all sufficiently large $N$.

The following counter examples show that neither the torsion-free hypothesis on $H$ nor the Scott--Swarup hypothesis on the extension can be dropped from \Cref{th:intro1}.

\begin{counterexample}[Lack of Scott--Swarup property]
In \cite[Section 1.1.1]{brinkmann2002splittings}, Brinkmann builds a hyperbolic automorphism $\phi$ of the free group $F= F_m * \langle a_1, \ldots a_{n}\rangle$, where $m \ge 3$, of the form
\begin{align*}
\phi(F_m) &= F_m,\\
\phi(a_i) &= 
	\begin{cases}
            a_{i+1} \quad \text{if} \;\; 1 \le i < n \\
            wa_1v \quad \text{if} \; \; i =n,
         \end{cases}
\end{align*}
where $w,v \in F_m$. 
Notice that the induced extension $G_\phi = F \rtimes_\phi \Z$ does not have the Scott--Swarup property: $F_m$ is not quasiconvex in $F_m \rtimes_\phi \Z$ (which is hyperbolic) and hence not quasiconvex in $G_\phi$. Focusing on the case where $n=2$, one sees that for each $k$ odd, $\phi^k$ has the property that $\phi^k(a_1) = w_ka_2v_k$ and $\phi^k(a_2) = w'_ka_1v'_k$ for some  $w_k,v_k,w'_k,v'_k \in F_m$. Hence, when $k$ is odd, $G_{\phi^k}$ is generated by $F_m$,  $a_1$, and a generator of $\Z$, making its rank at most $m+2 < \rank(F) +1$.
\end{counterexample}

\begin{counterexample}[Torsion in $H$]
Here we exploit the failure of \Cref{lem:big_lem}.\ref{cor:qc_or_big} in the presence of torsion.
Fix $m\ge 3$ and a prime $q$, and let $F,\: Z \le H$ be the groups 
\[H = \langle a_1,\dotsc, a_m, s \mid s^q =1,  [a_i,s] = 1\,\forall i\rangle,\quad F = \langle a_1,\dotsc, a_m\rangle \le H, \quad\text{and}\quad Z = \langle s\rangle \le H.\]
Thus $F$ is free with $\rank(F) =m$ and $H$ decomposes as a direct product $H = F\times Z$ with $\rank(H) = m+1$ and $[H:F] = q$. Let $\rho\colon H\to F$ denote the projection onto the $F$ factor  and $\iota\colon F\to H$ the inclusion of $F$ into $H$. Let $\beta\colon H\to H$ be the homomorphism defined by the assignments
\[
\begin{array}{ccccc}
\beta(s) = s &&\text{and}&& \beta(a_i ) = a_i s\quad\text{for each $i=1,\dotsc, m$}.
\end{array}
\]
Observe that $\beta^q$ is the identity, thus $\beta$ is in fact an automorphism of $H$. Since $\beta(a_1) \notin F$, we have $\beta(F) \ne F$.
Let $\tau \in \Aut(F)$ be any fully irreducible and atoroidal automorphism. Using the product structure of $H$, set $\alpha = \tau \times \mathrm{id}_Z$. We note that $\alpha$ is an automorphism of $H$ and that
\[\rho\iota = \mathrm{id}_F,\qquad\rho\beta = \rho,\qquad \rho\alpha = \tau\rho,\qquad\text{and}\qquad \rho\alpha\beta=\tau\rho. \]

Now let $\phi = \alpha\beta\in \Aut(H)$ and consider the extension $G = H \rtimes_\phi \Z$. Notice that $\phi(F) \ne F$. However, since $H$ contains only finitely many index $q$ subgroups, we may choose $n > 1$ so that $\phi^n(F) = F$. Let $G_n \le G$ be the preimage of $n\Z$ under the projection $G\to \Z$; this is an index $n$ subgroup with $G_n \cong H\rtimes_{\phi^n}\Z$. The further subgroup $G'_n \cong F\rtimes_{\phi^n\vert_F}\Z$ has $[G_n:G'_n] = q$. Since $\phi^n\vert_F = \rho\phi^n\iota = \tau^n$ is fully irreducible and atoroidal, $G'_n$ is hyperbolic \cite{BF92} and each finitely generated infinite index subgroup of $F$ is quasiconvex in $G'_n$ \cite{mitra1999theorem}. Since $[G:G'_n]$ is finite, our extension $G = H\rtimes_\phi \Z$ is also hyperbolic and has the Scott--Swarup property. However, the extension $G$ does not satisfy the conclusion of \Cref{th:intro1}: For all $k\ge 1$, the observation $\langle F, \phi^{kn+1}(F)\rangle=  H$ implies that the subextension $H\rtimes_{\phi^{kn+1}}\Z$ is generated by $F$ and the stable letter and thus rank at most $m+1 <  \rank(H) + 1$.
\end{counterexample}

\noindent \textbf{Acknowledgments}: 
This work drew inspiration from Souto's paper \cite{souto2008rank} and owe's an intellectual debt to the powerful machinery provided by Kapovich and Weidmann \cite{KapovichWeidman--indecomp,KapovichWeidmann-kleinian}. We thank the referee for helpful suggestions.

\section{Setup}
\label{sec:setup}

Fix a group $G$ with a finite, symmetric generating set $S$ and let $X = \cay{S}{G}$ be its Cayley graph. Equip $X$ with the path metric $d$ in which each edge has length $1$, making $(X,d)$ into a proper, geodesic metric space. 
For subsets $A,B\subset X$, define $d(A,B) = \inf\{d(a,b) \mid a\in A,b\in B\}$ and declare the \define{$\epsilon$--neighborhood} of $A$ to be $\nbhd{\epsilon}{A} = \{x\in X \mid d(\{x\},A) < \epsilon\}$. The \define{Hausdorff distance} between sets is defined as
\[\dhaus(A,B) = \inf\{\epsilon > 0 \mid  A\subset \nbhd{\epsilon}{B}\text{ and }B\subset \nbhd{\epsilon}{A}\}.\]

We identify $G$ with the vertices of $X$ and define the \define{wordlength} of $g\in G$ by $\abs{g}_S = d(e,g)$, where $e$ is the identity element of $G$. 
A \define{tuple} in $G$ is a (possibly empty) ordered list $L =(g_1,\dotsc,g_n)$ elements of $g$. The \define{length} of a tuple $L=(g_1,\dotsc,g_n)$ is the number $\len{L} = n$ of entries of the list, and its \define{magnitude} is defined to be $\magnitude{L} = \max_{i}\abs{g_i}_S$; for $h\in G$ we denote the tuple $(hg_1h\inv, \dotsc, h g_n h\inv)$ by $hLh\inv$. 
We define the \define{conjugacy magnitude} of a tuple $L$ to be $\conjmag{L} = \min_{h\in G}\norm{hLh\inv}$. 
The following three operations are called \define{elementary Nielsen moves} on a tuple $L= (g_1,\dotsc,g_n)$:
\begin{itemize}
\item For some $i\in \{1,\dotsc,n\}$, replace $g_i$ by $g_i\inv$ in $L$.
\item For some $i,j\in\{1,\dotsc, n\}$ with $i\ne j$, interchange $g_i$ and $g_j$ in $L$.
\item For some $i,j\in\{1,\dotsc, n\}$ with $i\ne j$, replace $g_i$ by $g_ig_j$ in $L$.
\end{itemize}
Two tuples are \define{Nielsen equivalent} if one may be transformed into the other via a finite chain of elementary Nielsen moves. Nielsen proved that any two minimal generating sets of a finitely generated free group are Nielsen equivalent \cite{nielsen1924isomorphismengruppe}. Hence, two tuples $L_1$ and $L_2$ of length $n$ are Nielsen equivalent if and only if there is an automorphism $\psi \colon F_n \to F_n$ such that $\phi_1 = \phi_2 \circ \psi$, where $\phi_i\colon F_n\to G$ is the homomorphism taking the $j$th element of a (fixed) basis for $F_n$ to the $j$th element of $L_i$. Note that Nielsen equivalent tuples generate the same subgroup of $G$. 

Following Kapovich--Weidmann \cite[Definition 6.2]{KapovichWeidmann-kleinian}, we consider the following variation:

\begin{definition}
\label{def:partitioned}
A \define{partitioned tuple} in $G$ is a list $M = (Y_1,\dotsc,Y_s;T)$ of tuples $Y_1,\dotsc,Y_s,T$ of $G$ with $s\ge 0$ such that (1) either $s > 0$ or $\len{T} > 0$, and (2) $\langle Y_i\rangle \ne \{e\}$ for each $i > 0$. Thus $(;T)$ (where $\len{T}>0$) and $(Y_1;)$ (where $\langle Y_1 \rangle \ne\{e\}$) are examples of partitioned tuples. The length of $M$ is defined to be $\len{M} = \len{Y_1}+\dotsb+ \len{Y_s}+\len{T}$. The \define{underlying tuple} of $M$ is the $\len{M}$--tuple $\underlie{M}=(Y_1,\dotsc,Y_s,T)$ obtained by concatenating $Y_1,\dotsc,Y_s,T$.
The \define{elementary moves} on a partitioned tuple $M = (Y_1,\dotsc, Y_s; (t_1,\dotsc, t_n))$ consist of:
\begin{itemize}
\item For some $i\in \{1,\dotsc, s\}$ and $g\in \langle (\cup_{j\ne i}Y_j)\cup \{t_1,\dotsc, t_n\}\rangle$, replace $Y_i$ by $g Y_i g\inv$.
\item For some $k\in \{1,\dotsc, n\}$ and elements $u,u'\in \langle (\cup_j Y_j)\cup \{t_1,\dotsc, t_{k-1},t_{k+1},\dotsc, t_n\}\rangle$, replace $t_k$ by $u t_k u'$.
\end{itemize}
Two partitioned tuples $M$ and $M'$ are \define{equivalent} if $M$ can be transformed into $M'$ via a finite chain of elementary moves. In this case, it is easy to see that the underlying tuples $\underlie{M}$ and $\underlie{M'}$ are Nielsen equivalent.

\end{definition}

We henceforth assume that $G$ is a \define{hyperbolic group}, which is equivalent to requiring that $X$ be \define{$\delta$--hyperbolic} for some fixed $\delta \ge 0$. This means that every geodesic triangle $\triangle(a,b,c)$ in $X$ is \define{$\delta$--thin} in the sense that each side is contained in the $\delta$--neighborhood of the union of the other two.
A \define{geodesic} in $X$ is a map $\gamma\colon \J\to X$ of an interval $\J\subset \R$ such that $\abs{s-t} = d(\gamma(s),\gamma(t))$ for all  $s,t\in \J$.
Two geodesic rays $\gamma_1,\gamma_2\colon \R_+\to X$ are \define{asymptotic} if $\dhaus(\gamma_1(\R_+),\gamma_2(\R_+)) < \infty$. The \define{Gromov boundary} of $X$ is defined to be the set $\partial X$ of equivalence classes of geodesic rays in $X$. Note that every isometry of $X$ induces a self-bijection of $\partial X$. The equivalence class or \define{endpoint} of a ray $\gamma\colon \R_+\to X$ is denoted $\gamma(\infty)\in \partial X$, and $\gamma$ is said to \define{join} $\gamma(0)$ to $\gamma(\infty)$.
A biinfinite geodesic $\gamma\colon \R\to X$ determines two rays and is said to \define{join} their respective endpoints $\gamma(-\infty)$ and $\gamma(\infty )$. 
The fact that $X$ is proper and $\delta$--hyperbolic ensures that any two points of $X\cup\partial X$ can be joined by a geodesic segment, ray, or line; see \cite{KapovichBenakli-boundaries,KapovichWeidman--indecomp}. 
The \define{convex hull} of a set $Y\subset X\cup \partial X$ is the union $\conv(Y)$ of all geodesics joining points of $Y$ (including degenerate geodesics of the form $\{0\}\to Y$). The set $Y$ is \define{$\epsilon$--quasiconvex} if $\conv(Y)\subset \nbhd{\epsilon}{Y}$. A subgroup $U\le G$ is $\epsilon$--quasiconvex if it is so when viewed as a subset of $X$.
We refer the reader to \cite{Gromov,GhysdelaHarpe,BH} for further background on hyperbolic groups.

A sequence $\{x_n\}$ in $X$ is said to \define{converge} to $\zeta\in \partial X$ if for some (equivalently every) geodesic $\gamma\colon \R_+\to X$ in the class $\zeta$ and sequence $\{t_m\}$ in $\R_+$ with $t_m\to \infty$, one has
\[\lim_{n,m} \big(d(x_n, x_0) + d(\gamma(t_m),x_0) - d(x_n, \gamma(t_m))\big) = \infty.\]
The \define{limit set} of a subgroup $U\le G$ is the set $\Lambda(U)$ accumulation points $\zeta\in \partial X$ of an orbit $U\cdot x_0\subset X$; the fact that any two orbits of $U$ have finite Hausdorff distance implies that this is independent of the point $x_0$. Following Kapovich--Weidmann \cite[Definition 4.2]{KapovichWeidman--indecomp} we define the \define{hull} of a subgroup $U$ to be
\[\hull{U} = \overline{\conv\Big(\conv\big(\Lambda(U)\cup \{x\in X\mid d(x,u\cdot x) \le 100\delta\text{ for some $u\in U\setminus\{e\}$}\}\big)\Big)}.\]

We leave the following fact as an exercise for the reader. Alternatively, it follows from a slight modification of \cite[Lemma 4.10 and Lemma 10.3]{KapovichWeidman--indecomp}.

\begin{lemma}
\label{lem:haus-dist-subgroups}
There is a constant $A=A(\epsilon)$ for each $\epsilon\ge 0$ such that $\dhaus(U,\hull{U})\le A$ for every torsion-free $\epsilon$--quasiconvex subgroup $U$ of $G$.
\end{lemma}
%\begin{proof}[Proof sketch]
%Since $U$ is torsion-free, each $g\in U$ acts as a loxodromic isometry of $X$ and thus has two fixed points $g_{\pm}\in \partial X$. Let $\gamma\colon \R\to X$ be a biinfinite geodesic joining $g_-$ to $g_+$. It follows from basic hyperbolic geometry that $\{x\in X\mid d(x,g\cdot x)\le 100\delta\}$ is contained in a uniform neighborhood of $\gamma(\R)$ (probably something like $\nbhd{200\delta}{\gamma}$). Then one should be able to argue that $\hull{U}$ is contained in a uniform neighborhood of $\conv(\Lambda(U))$ and conversely. Thus $\hull{U}$ and $\conv(\Lambda(U))$ have uniformly bounded Hausdorff distance.
%
%Next, since $U$ is $\epsilon$--quasiconvex, you can see that $\conv(\Lambda(U))$ is contained in some $c=c(\epsilon)$ neighborhood of $U$: given a biinfinite geodesic connecting distinct $\zeta_1,\zeta_2\in \Lambda(U)$, approximate it as a limit of geodesic segments connecting $x_{1,n}$ to $x_{2,n}$, where $x_{i,n}$ converges to $\zeta_i$. Quasiconvexity now shows that $\conv(\Lambda(U))$ is contained in a neighborhood of $U$, as claimed. Conversely, we have to show each $u\in U$ is contained near $\conv(\Lambda(U))$, but there should be some standard argument for this.
%\end{proof}

By noting that there are only finitely many subgroups of $G$ that may be generated by elements from the finite set $\nbhd{r}{\{e\}}$, we have the following lemma:
\begin{lemma}
\label{lem:uniorm_qc}
There is a constant $c = c(r)$ for each $r> 0$ such that every quasiconvex subgroup $U\le G$ generated by elements from the $r$--ball $\nbhd{r}{\{e\}}$ is $c$--quasiconvex.
\end{lemma}

The following technical result of Kapovich and Weidmann is a key ingredient in our argument:

\begin{theorem}[Kapovich--Weidmann {\cite[Theorem 6.7]{KapovichWeidmann-kleinian}}, c.f. {\cite[Theorem 2.4]{KapovichWeidman--indecomp}}]
\label{thm:KapWeid}
For every $m\ge 1$ there exists a constant $K=K(m)\ge 0$ with the following property. Suppose that $M = (Y_1,\dotsc, Y_s;T)$ is a partitioned tuple in $G$ with $\len{M} = m$ and let $H = \langle\underlie{M}\rangle$ be the subgroup generated by the underlying tuple of $M$. Then either
\[H = \langle Y_1 \rangle \ast \dotsb \ast \langle Y_s \rangle \ast \langle T\rangle,\]
with $\langle T\rangle$ free on the basis $T$, or else $M$ is equivalent to a partitioned tuple $M' = (Y'_1,\dotsc, Y'_s; T')$ for which one of the following occurs:
\begin{enumerate}
\item There are $i,j\in \{1,\dotsc,s\}$ with $i\ne j$ and $d(\hull{\langle Y'_i\rangle},\hull{\langle Y'_j\rangle})\le K$.
\item There is some $i\in \{1,\dotsc, s\}$ and $t\in T'$ such that $d(\hull{\langle Y'_i\rangle}, t \cdot \hull{\langle Y'_i\rangle})\le K$.
\item There exists an element $t\in T'$ with a conjugate in $G$ of wordlength at most $K$.
\end{enumerate}

\end{theorem}

We conclude this section with the following lemma, which ties into the conclusions of \Cref{thm:KapWeid} and is an adaptation of \cite[Propositions 7.3--7.4]{KapovichWeidmann-kleinian} to our context. Since the hypotheses of \cite{KapovichWeidmann-kleinian} are not satisfied here, we include a short proof.

\begin{lemma}
\label{lem:combine}
For every $K,r>0$ there is a constant $B=B(K,r)$ with the following property: Let $Y_1,Y_2,Y_3$ be tuples in $G$ generating torsion-free quasiconvex subgroups $U_i = \langle Y_i\rangle$ and satisfying $\conjmag{Y_i}\le r$ for each $i=1,2,3$.
\begin{itemize}
\item If $d(\hull{U_1},\hull{U_2})\le K$, then $(Y_1,Y_2)$ is Nielsen equivalent to a tuple $Y$ satisfying $\conjmag{Y}\le B$.
\item If $d(\hull{U_3},g\cdot\hull{U_3})\le K$ for $g\in G$, then $(Y_3,(g))$ is Nielsen equivalent to a tuple $Z$ with  $\conjmag{Z}\le B$.
\end{itemize}
\end{lemma}
\begin{proof} 
For brevity, we prove the claims simultaneously. By assumption, we may choose points $x_1\in \hull{U_1}$, $x_2\in \hull{U_2}$ and $z_3,z_4\in \hull{U_3}$ with $d(x_1,x_2)\le K$ and $d(z_3,gz_4)\le K$.
For $i=1,2,3$, we also choose $h_i\in G$ such that $\magnitude{h_iY_ih_i\inv}\le r$. The subgroups $U'_i = h_i U_i h_i\inv$ are then $c(r)$--quasiconvex by \Cref{lem:uniorm_qc} and hence satisfy $\dhaus(U_i',\hull{U_i'})\le A(c(r))$ by \Cref{lem:haus-dist-subgroups}.
Noting that $\hull{U'_i} = h_i \hull{U_i}$, we may choose $u_i\in U_i$ for $i=1,2$ such that $d(h_i u_i h_i\inv , h_ix_i)\le A(c(r))$. Similarly choose $w_j\in U_3$ so that $d(h_3w_jh_3\inv, h_3 z_j)\le A(c(r))$ for $j=3,4$.
Set $B= 4A(c(r))+2K+r$. 

To conclude the second claim, observe that
\begin{align*}
\abs{h_3 (w_3\inv g w_4) h_3\inv}_S &= d(w_3 h_3\inv , gw_4h_3\inv) \\
&\le d(w_3 h_3\inv, z_3) + d(z_3, gz_4) + d(gz_4, g w_4 h_3\inv) \\
&\le 2A(c(r)) + K.
\end{align*}
Since $\magnitude{h_3Y_3h_3\inv}\le r$ as well, the concatenated tuple $Z = (Y_3,(w_3\inv g w_4))$ clearly satisfies $\conjmag{Y'} \le B$. Further, since $w_3,w_4\in \langle Y_3\rangle$, it is immediate that $Z$ is Nielsen equivalent to $(Y_3,(g))$.

For the first claim, set $f = h_1u_1\inv u_2 h_2\inv$ and use the triangle inequality to observe
\begin{align*}
\abs{f}_S 
&= d(u_1h_1\inv,u_2h_2\inv)\\
&\le d(u_1h_1\inv, x_1) + d(x_1,x_2) + d(x_2,u_2h_2\inv)\\
&\le 2A(c(r)) + K.
\end{align*}
Since $\magnitude{h_2Y_2h_2\inv}\le r$, another use of the triangle inequality gives
\[\magnitude{h_1(u_1\inv u_2 Y_2 u_2\inv u_1) h_1\inv} = \magnitude{f (h_2 Y_2 h_2\inv) f\inv} \le 4A(c(r))+2K+r = B.\]
The concatenated tuple $Y = (Y_1, u_1\inv u_2 Y_2 u_2 u_1\inv)$ thus evidently satisfies $\conjmag{Y}\le B$. To complete the proof, it only remains to show that $(Y_1,Y_2)$ is Nielsen equivalent to $Y$. But this is clear: since $u_2\in \langle Y_2\rangle$  the tuple $(Y_1, Y_2)$ is equivalent to $(Y_1, u_2 Y_2 u_2\inv)$ which, since $u_1\inv\in \langle Y_1\rangle$, is in turn equivalent to $Y$. 
\end{proof}

\section{Proof of the main result}
Suppose now that our fixed group $G$ fits into a short exact sequence 
\begin{align} \label{eq:sec}
1 \longrightarrow H \longrightarrow G \overset{p}{\longrightarrow} \Gamma \longrightarrow 1
\end{align}
of infinite hyperbolic groups
that enjoys the Scott--Swarup property with torsion-free kernel $H$.  Recall that the conjugation action of $G$ on $H$ induces a homomorphism $\Phi \colon \Gamma \to \Out(H)$ and that, since $G$ is hyperbolic, $\Phi$ has finite kernel. For any subgroup $\Delta \le \Gamma$, we set $G_\Delta = p^{-1}(\Delta) \le G$, and note that this subgroup of $G$ fits into the sequence $1 \to H \to G_\Delta \to \Delta \to 1$. 

The follow lemma summarizes some of the basic properties we will require.

\begin{lemma} \label{lem:big_lem}
For the sequence \eqref{eq:sec}, we have the following:
\begin{enumerate}\renewcommand{\theenumi}{\roman{enumi}}
\item For every infinite order $g\in \Gamma$, $\Phi(g) \in \Out(H)$ does not fix the conjugacy class of any infinite index, finitely generated subgroup of $H$. \label{lem:no_fix}
\item The kernel $H$ is either free of rank at least $3$ or else isomorphic to the fundamental group of a closed surface of genus at least $2$. \label{lem:freesurface}
\item Every proper subgroup $U\lneq H$ is either quasiconvex in $G$ or else has $\rank(U)>\rank(H)$. \label{cor:qc_or_big}
\end{enumerate}
\end{lemma}

\begin{proof}
To prove item (\ref{lem:no_fix}), suppose towards a contradiction that $g\in\Gamma$ of infinite order fixes the conjugacy class of an infinite index, finitely generated subgroup $A$ of $H$. Then, after applying an inner automorphism of $H$, we see that the semidirect product $A \rtimes_\phi \Z$ is contained in $G$, where $\phi$ is an automorphism in the class $\Phi(g)$. However, it is well-known that the subgroup $A$ is distorted (i.e. not quasi-isometrically embedded) in $A \rtimes_\phi \Z$ and hence distorted in $G$. This, however, contradicts the Scott--Swarup property and proves item (\ref{lem:no_fix}).

Next, the theory of JSJ decompositions for hyperbolic groups \cite{RipsSela} (see also \cite{Levitt}) shows that a sequence of hyperbolic groups as in \eqref{eq:sec} with torsion-free kernel $H$ must have $H$ isomorphic to the free product $(\ast_{i=1}^k\Sigma_i)\ast\free_n$, where $\free_n$ is free of rank $n$ and each $\Sigma_i$ is the fundamental group of a closed surface. We must show that this factorization is trivial, i.e. either $k=0$ or $n=0$. This follows from the fact that such a  nontrivial free product decomposition is canonical (e.g. \cite[Theorem 3.5]{scott1979topological}) and so is preserved under \emph{any} automorphism of $H$ (up to permuting the factors). Hence, for each infinite order $g \in \Gamma$, some power of $\Phi(g)$ fixes the conjugacy class of a surface group factor of $H$, contradicting item (\ref{lem:no_fix}) above unless $k=0$ or $n=0$. This proves (\ref{lem:freesurface}).

For (\ref{cor:qc_or_big}), let $J = [U:H] > 1$. If $J = \infty$, then $U$ is quasiconvex in $G$ by the Scott--Swarup property. Otherwise basic covering space theory implies $\rank(U) = m(1-J) + J\rank(H)$ for $m \in\{1,2\}$ depending, respectively, on whether $H$ is free or the fundamental group of a closed surface. 
\end{proof}

The following lemma is essential proven in \cite[Corollary 11]{kapovich2000hyperbolic} in the case where $H$ is free and $\Gamma$ is cyclic. We sketch the argument for the reader.

\begin{lemma}
\label{lem:no_split}
If $1\to H\to G\to \Gamma\to 1$ is a sequence of infinite hyperbolic groups such that $H$ is torsion-free and $G$ has the Scott--Swarup property, then $G$ does not split over a cyclic (or trivial) group.
Moreover, the same holds for $G_\Delta \le G$ whenever the subgroup $\Delta \le \Gamma$ is infinite.
\end{lemma}

\begin{proof}
We prove the moreover statement since it is clearly stronger.
Let $\Delta \le \Gamma$ be an infinite subgroup.
Suppose towards a contradiction that $G_\Delta$ has a minimal, nontrivial action on a simplicial tree $T$ with cyclic (or trivial) edge stabilizers. Since $H$ is normal in $G_\Delta$, the action $H \curvearrowright T$ is also minimal. Hence the main theorem of \cite{bestvina1991bounding}, implies that $T/H$ is a finite graph. Notice that $\Delta$ acts on the corresponding graph of groups decomposition of $H$ (via $\Phi \colon \Gamma \to \Out(H)$). First, this decomposition must have trivial edge groups: an infinite cyclic edge stabilizer would be fixed under some infinite order $g\in \Delta \le \Gamma$, contradicting that $G$ is hyperbolic. 
Hence, the nontrivial graph of groups $T/H$ has trivial edge stabilizers, but this implies that $\Delta$ virtually fixes this splitting of $H$. From this we obtain an infinite order element $g\in \Delta \le \Gamma$ which fixes a vertex group $A$ of the splitting. Since $A$ is finitely generated and has infinite index in $H$, we have a contradiction to \Cref{lem:big_lem}.\ref{lem:no_fix}. This completes the proof.
\end{proof}

Let us establish notation and specify the constants for the proof \Cref{th:intro1}. Let $\bar{S}\subset \Gamma$ be the image of our fixed generating set $S\subset G$. We assume that $\ell_\Gamma(\cdot)$ is conjugacy length in $\Gamma$ with respect to $\bar{S}$. For the given $r$, let $K$ be the maximum of the constants $K(1),\dotsc, K(\rank(H)+r)$ provided by \Cref{thm:KapWeid}. Set $D_0 = K$ and use \Cref{lem:combine} recursively to define $D_{n+1} = \max\{D_n,B(K,D_n)\}$ for each $n\in \N$. Set $N = 1+ D_{2\rank(H)}$ and suppose that $\Delta\le\Gamma$ is any subgroup with $\rank(\Delta)\le r$ and $\ell_\Gamma(\delta) \ge N$ for all $\delta\in \Delta\setminus\{1\}$. Let $G_\Delta$ be the preimage of $\Delta$ under the projection $p \colon G\to \Gamma$. We make the following observations:
\begin{lemma}
\label{claim:short_conj_in_kernel}
If $Y$ is a tuple in $G$ with $Y\subset G_\Delta$ and $\conjmag{Y} < N$, then $\langle Y\rangle \le H$. 
\end{lemma}
\begin{proof}
Choose $g\in G$ so that $\magnitude{gYg\inv} < N$. Then for each $y\in Y$ we have
\[ \abs{p(g)p(y)p(g)\inv}_{\bar{S}} = \abs{p(gyg\inv)}_{\bar{S}} \le \abs{gyg\inv}_S < N\]
which shows that $\ell_\Gamma(p(y)) < N$. Since we also have $p(y)\in \Delta$ by assumption, this gives $p(y) = 1$ and hence $y\in H$ by the hypothesis on $\Delta$. Thus $\langle Y \rangle \le H$.
\end{proof}

\begin{lemma}
\label{claim:induct}
Fix $n\in \{0,\dotsc,2\rank(H)-1\}$ and suppose that $M = (Y_1,\dotsc,Y_s;T)$ is a partitioned tuple with $\langle\underlie{M}\rangle = G_\Delta$ and $\len{M}\le (\rank(H)+r)$ such that for each $i\in \{1,\dotsc s\}$ we have $\conjmag{Y_i}\le D_n$ with $\langle Y_i \rangle $ quasiconvex.
Then there is a partitioned tuple $\tilde{M} = (\tilde{Y}_1,\dotsc \tilde{Y}_{\tilde{s}}; \tilde{T})$ satisfying $\conjmag{\tilde{Y}_j}\le D_{n+1}$ for each $j\in \{1,\dotsc, \tilde{s}\}$ such that $\underlie{\tilde{M}}$ is Nielsen equivalent to $\underlie{M}$ and either 
\begin{enumerate}\renewcommand{\theenumi}{\alph{enumi}}
\item\label{option:shorten} $\len{\tilde{T}} < \len{T}$ with $\tilde{s}\le s+1$ or else
\item\label{option:combine} $\len{\tilde{T}} = \len{T}$ with $\tilde{s} < s$.
\end{enumerate}
\end{lemma}
\begin{proof}
Since $\len{M} \le \rank(H)+r$ and $\langle\underlie{M}\rangle =  G_\Delta$ does not split as a nontrivial free product (\Cref{lem:no_split}), we may apply \Cref{thm:KapWeid} to obtain a partitioned tuple $M' =(Y'_1,\dotsc, Y'_s; T')$ that is equivalent to $M$ and satisfies one of the three conclusions of that theorem. Since all elementary moves on a partitioned tuple $(W_1,\dotsc, W_p;V)$ preserve the conjugacy class of each tuple $W_i$, we have $\conjmag{Y'_i} \le D_n$ with $\langle Y'_i\rangle$ quasiconvex for each $i$. As $D_n < N$, \Cref{claim:short_conj_in_kernel} gives $\langle Y'_i\rangle \le H$ and so ensures that $\langle Y'_i\rangle$ is torsion-free.

We now analyze the conclusions of \Cref{thm:KapWeid}: If $M'$ satisfies conclusion (1), then after reordering we may assume $d(\hull{\langle Y'_1\rangle},\hull{\langle Y'_2\rangle})\le K$ and use \Cref{lem:combine} to find a tuple $Y$ Nielsen equivalent to ($Y'_1,Y'_2)$ with $\conjmag{Y}\le D_{n+1}$. The partitioned tuple $(Y,Y'_3,\dotsc, Y'_{s};T')$ then satisfies the claim. If $M$ satisfies (2), then after reordering we have $d(\hull{\langle Y'_1\rangle},t\cdot\hull{\langle Y'_1\rangle}) \le K$ for some $t\in T'$ and so may use \Cref{lem:combine} to find a tuple $Z$ equivalent to $(Y'_1,(t))$ with $\conjmag{Z}\le D_{n+1}$. Here we take $\tilde{M} = (Z, Y'_2,\dotsc, Y'_s; T'\setminus\{t\})$ to complete the claim. If $M'$ satisfies (3), then $T'$ contains an element $t$ with $\conjmag{(t)}\le K\le D_{n+1}$ and the partitioned tuple $(Y'_1,\dotsc, Y'_s, (t); T'\setminus\{t\})$ satisfies the claim.
\end{proof}

The pieces are now in place to prove our main theorem:
\begin{proof}[Proof of \Cref{th:intro1}]
Let $L$ be any minimal-length tuple with $\langle L \rangle =G_\Delta$. Since $G_\Delta$ has a standard generating set of size  $\rank(H) + \rank(\Delta)$, we have $\len{L}\le \rank(H)+r$. Set $M_0 = (;L)$ and observe that $M_0$ satisfies \Cref{claim:induct} with $n=0$. We may therefore inductively apply \Cref{claim:induct} (with $n=0, 1,\dotsc$) to obtain a sequence $M_0,M_1,\dotsc$ of partitioned tuples each with $\underlie{M_i}$ Nielsen equivalent to $L$. 
After inducting as many times as possible, we obtain a partitioned tuple $M_k = (Y_1,\dotsc, Y_s; T)$ that satisfies $\conjmag{Y_i}\le D_k$ for each $i$ (by construction) but violates the hypotheses of  \Cref{claim:induct}, either because $k = 2\rank(H)$ or because some $\langle Y_i\rangle$ fails to be quasiconvex.
Since $\conjmag{Y_i}\le D_k < N$, \Cref{claim:short_conj_in_kernel} ensures that $\langle Y_i\rangle \le H$ for each $i$. Since $G_\Delta = \langle \underlie{M_k}\rangle$ surjects onto $\Delta$, it follows that $\len{T}\ge \rank(\Delta)$. Thus at most $\len{L}-\rank(\Delta)$ applications of \Cref{claim:induct} could have reduced the length of $T$ (option \ref{option:shorten}) and so at least $k-\len{L}+\rank(\Delta)$ applications must have combined $Y_i$'s (option \ref{option:combine}). It now follows that $k < 2\rank(H)$, for otherwise $k$ applications of the claim would necessarily produce a tuple $Y_i$ with $\len{Y_i} > \rank(H)$, contradicting $\len{Y_i} + \len{T} \le \rank(H)+\rank(\Delta)$. 

Since $M_k$ violates \Cref{claim:induct} but $k < 2\rank(H)$, it must be that some $\langle Y_i\rangle$ fails to be quasiconvex. After reordering, let us assume $\langle Y_1\rangle \le H$ is not quasiconvex. Note that we also cannot have $\rank(\langle Y_i\rangle) > \rank(H)$, for otherwise $\len{Y_i}+\len{T} > \rank(H)+\rank(\Delta)$ contradicting our choice of $L$. The only possibility afforded by \Cref{lem:big_lem}.\ref{cor:qc_or_big} is therefore $\langle Y_1 \rangle =  H$ with $\len{Y_1} = \rank(H)$. Since $\len{M_k} \le \rank(H)+\rank(\Delta)$, it follows that $M_k$ is of the form $M_k=(Y_1;T)$ with $\len{Y_1} = \rank(H)$ and  $\len{T} = \rank(\Delta)$. Therefore $M_k$ is a standard generating set for $G_\Delta$ that is Nielsen equivalent to $L$.
\end{proof}

\bibliographystyle{alphanum}
\bibliography{rank_hyperbolic_extensions}

\bigskip

\noindent
\begin{minipage}{.55\linewidth}
Department of Mathematics\\
Vanderbilt University\\
1326 Stevenson Center\\
Nashville, TN 37240, USA\\
E-mail: {\tt spencer.dowdall@vanderbilt.edu}
\end{minipage}
\begin{minipage}{.45\linewidth}
Department of Mathematics\\ 
Temple University\\ 
%Wachman Hall\\
1805 North Broad Street \\
Philadelphia, PA 19122, USA\\
E-mail: {\tt samuel.taylor@temple.edu}
\end{minipage}

\end{document}